\documentclass[10pt]{amsart}
\usepackage{amssymb,amsmath,enumerate}  
\usepackage{textpos}
\title{Korenblum-Type Extremal Problems in Bergman Spaces}  

\usepackage{tikz}
\usepackage{graphicx}
\usepackage{booktabs}
\usepackage{adjustbox}

\author[P.Chakraborty]{Pritha Chakraborty}
\address{Department of Mathematics and Statistics, Texas Tech
University, Lubbock, TX 79409} \email{pritha.chakraborty@ttu.edu}
\author[A.Solynin]{Alexander Solynin}
\address{Department of Mathematics and Statistics, Texas Tech
University, Lubbock, TX 79409} \email{alex.solynin@ttu.edu}

\newtheorem{definition}{Definition}[section]
\newtheorem{theorem}[definition]{Theorem}
\newtheorem{lemma}[definition]{Lemma}
\newtheorem{corollary}[definition]{Corollary}

\newtheorem{conjecture}{Conjecture}
\newtheorem{lemmaA}{Lemma}

\newtheorem{problem}{Problem}

\newcommand{\blemmaA}{\begin{lemmaA}}
\newcommand{\elemmaA}{\end{lemmaA}}
\newtheorem{proposition}[definition]{Proposition}
\newcommand{\bprop}{\begin{proposition}}
\newcommand{\eprop}{\end{proposition}}

\newcommand{\bt}{\begin{theorem}}
\newcommand{\et}{\end{theorem}}
\newcommand{\bc}{\begin{corollary}}
\newcommand{\ec}{\end{corollary}}
\newcommand{\bcon}{\begin{conjecture}}
\newcommand{\econ}{\end{conjecture}}

\begin{document}%


\begin{abstract}%
We shall study non-linear extremal problems in Bergman space $\mathcal{A}^2(\mathbb{D})$. We show the existence of the solution and that the extremal functions are bounded. Further, we shall discuss special cases for polynomials, investigate the properties of the solution and provide a bound for the solution. This problem is an equivalent formulation of B. Korenblum's conjecture, also known as Korenblum's Maximum Principle: for $f$, $g\in \mathcal{A}^2(\mathbb{D})$, there is a constant $c$, $0<c<1$ such that if $|f(z)|\leq |g(z)|$ for all $z$ such that $c<|z|<1$, then $\|f\|_2\leq \|g\|_2$. The existence of such $c$ was proved by W. Hayman but the exact value of the best possible value of $c$, denoted by $\kappa$, remains unknown.
\end{abstract}%

\maketitle%

\section{Korenblum's Maximum Principle: History and recent results}

Let $\mathbb{D}=\{z\in\mathbb{C}:\ |z|<1\}$ be the open unit disk and $A(c,1)=\{z\in\mathbb{C}:\, c<|z|<1\}$ be the annulus defined in the complex plane $\mathbb{C}$. Then, the Bergman space $\mathcal{A}^2(\mathbb{D})$ is the class of functions $f$ analytic in $\mathbb{D}$ with
\begin{eqnarray*}%
\| f \|_{2}= \left( \frac{1}{\pi} \int \limits_{\mathbb{D}} |f(z)|^2\; dA(z) \right)^{\frac{1}{2}} <\infty,
\end{eqnarray*}%
where $dA=r dr d\theta$ denotes Lebesgue area measure. Stefan Bergman developed this theory which was highly inspired from the related theory of Hardy spaces. An extensive study of Bergman spaces can be found in \cite{DS,HKZ,B1}.  
\par
The classical Maximum Modulus Theorem states that if a function $f$ is analytic in $\mathbb{D}$ and $|f(z)|\leq K$ in $A(c,1)$ for some constant $K$ and some fixed constant $c$, then $|f(z)|\leq K$ for all $z$ in $\mathbb{D}$. Hence $\|f\|_2\leq K=\|K\|_2$. Then it is quite natural to ask what happens if $K$ is replaced by any arbitrary non-constant analytic function. On this note, Boris Korenblum \cite{K1} conjectured in $1991$ that for $f,\ g \in \mathcal{A}^2(\mathbb{D})$, there is a constant $c$, $0<c<1$, such that if
\begin{equation*}\label{eq 2.1a}%
| f(z)| \leq |g(z)|,\quad z\in A(c,1)
\end{equation*}%
then
\begin{equation*}\label{eq 2.2a}%
\| f \|_2 \leq \| g \|_2
\end{equation*}%
In that paper, he proved a weaker version of this conjecture with an additional assumption that each zero of $f$ is a zero of $g$. It is easy to observe that if $g$ has no zeros in $\mathbb{D}$, then the quotient $f/g$ is analytic. Therefore, by the Maximum Modulus Theorem, $|f(z)|\leq |g(z)|$ in $\mathbb{D}$ which further implies that $\|f\|_2\leq \|g\|_2$. However, Hayman showed that this conclusion is not true in general if we replace $\mathbb{D}$ by $A(c,1)$ and if $g$ has a zero in $\mathbb{D}$. Precisely, he considered $f(z)=a$, $g(z)=z$ with $1/\sqrt{2}<a<c$ to show that the conclusion fails if $c>1/\sqrt{2}$. Therefore, this conjecture is also sometimes known as Korenblum's Maximum Principle or Bergman space Maximum Principle. Hayman \cite{H1} proved the conjecture in $1999$ with $c=0.04$. In this paper, we call the best possible value of such $c$ for which Korenblum's Maximum Principle is true for all functions in $\mathcal{A}^2(\mathbb{D})$ as {\it Korenblum's constant} and denote it by $\kappa$. The exact value of $\kappa$ is not yet known. Various partial results came in a sequence of papers by Korenblum, Richards, O.~Neil, Matero and Schwick \cite{K2,K3,J,SCH}. Towards finding the sharp value of $\kappa$, A. Hinkannen \cite{H2} improved the lower bound of $\kappa$ to $0.15173$. He also generalized the result in $\mathcal{A}^p(\mathbb{D})$ for $p\geq 1$. In addition, a series of examples obtained by Wang \cite{W1,W2,W3,W5,W6,W7} over the years have improved the upper bound of $\kappa$ to $0.6778994$. In recent papers of Wang \cite{W8,W6} the best known bounds to date can be found which are, $0.28185 <\kappa < 0.6778994$. 
 \par%
 Furthermore, it became quite natural to ask what happens if we replace (i) $\mathcal{A}^p(\mathbb{D})$ by $\mathcal{B}$, where $\mathcal{B}$ is the class of analytic functions in $\mathbb{D}$ with norm $\|f\|_{\mathcal{B}}$ and (ii) $A(c,1)$ by any arbitrary set $E\subset \mathbb{D}$. On this note, Hayman and Danikas \cite{HD} proved Korenblum-type results when (a) $\mathcal{B}$ is an Hardy space $\mathcal{H}^p$ for $0<p<\infty$ and $\mathcal{H}^{\infty}$, (b) $\mathcal{B}$ is the space of disk algebras. In addition, if $\mathcal{B}$ is a Fock space $F$, Schuster \cite{SC} proved that there is a positive constant $c$ with the property that whenever $f$ and $g$ are entire functions satisfying $|f(z)|\leq |g(z)|$ for $|z| > c$, then $\|f\|_F\leq \|g\|_F$ with $c=0.54$.
\par%
Pacing towards a slightly different direction, let us introduce the readers to the theory of extremal problems in Bergman spaces. Since Bergman space can be thought as an extension of Hardy space, analogous counterparts of problems in Hardy spaces are studied in Bergman spaces. The theory of general extremal problems in Hardy spaces is developed in the seminal works of S.~Ya.~Khavinson and Rogosinski \cite[Chapter 8]{D1}. On a similar note, minimal area problems have been extensively studied by Shapiro and Solynin \cite{AShS1,AShS3}. However, the standard techniques of functional analysis which worked quite smoothly for Hardy spaces failed heavily for Bergman spaces. There have been attempts to develop a theory of dual extremal problems for Bergman spaces and partial results were obtained in Vuk\'{o}tic, Khavinson and Stessin \cite{V1,KS}. Therefore, among many basic unsolved questions, the theory of extremal problems in Bergman spaces is still at a very beginning. On a brighter note, Sheil-Small in \cite{SH} recently solved the extremal problem of finding the explicit extremal function which minimizes the $\mathcal{A}^2$ norm for functions analytic and non-vanishing in $\mathbb{D}$ and of the form $f(z) = 1+a z + a_2z^2 +\ldots$, with $a\geq0$ given.
\par%
Let us briefly discuss the contents of our paper. In Section $2$, we introduce extremal problems \ref{prob 1}, \ref{prob 2} and \ref{prob 3} which can be thought of as the equivalent formulations of Korenblum's Maximum Principle for finding the best possible constant $\kappa$. In Section $3$, we review some of the preliminary results from the theory of Bergman spaces which we require for proving our results discussed in later sections. Our main results are demonstrated from Section $4$ onwards which are focussed primarily on Problems \ref{prob 1} and \ref{prob 2}. In Section $4$, we prove the existence of the extremal pair of functions which solves Problem \ref{prob 1} and \ref{prob 2} and further discuss the properties of the extremal pair. However, due to the complex nature of the non-linear functional and the absence of convexity in the functional, it has not been possible to comment on the uniqueness of the extremal pair of functions. In Section $5$, the properties of the extremal function are discussed thoroughly and in particular, we give the explicit bounds to extremal function of Problems \ref{prob 1} and \ref{prob 2}. In Section $6$, we discuss special cases of Problem \ref{prob 1} for polynomials, Blashcke products and bounded functions. We prove that the extremal pair for the general class of analytic functions is bounded and further solve the extremal problem explicitly for linear polynomials. In Section $7$, we solve the ``dual'' Problem A and leave Section $8$ for open questions to discuss.


\section{Versions of Korenblum's Maximum Principle}
\setcounter{equation}{0}  %

Consider the following general extremal problem.
\begin{problem}\label{prob 1}%
Given $0<c<1$,
\begin{enumerate}
\item[{\rm 1.}] Find%
\begin{equation}\label{eq 2.1}
F(c) := \inf \limits_{f,g\in S^1} \sup \limits_{c\leq |z|<1} \left |\frac{f(z)}{g(z)} \right |
\end{equation}%
where $S^1=\{ f\in \mathcal{A}^2(\mathbb{D}):\ \|f\|_2=1\}$ denotes a unit ball in Bergman space $\mathcal{A}^2(\mathbb{D})$.

\item[{\rm 2.}] Find an extremal pair of functions $f,g\in S^1$ such that $F(c)=\left|\frac{f(z_0)}{g(z_0)}\right|$ where $z_0\in(c,1)$.
\end{enumerate}

\end{problem}%
We call the function $F(c)$ introduced in (\ref{eq 2.1}) as {\it Korenblum's function for Problem \ref{prob 1}}. Every pair of functions $(f_0,g_0)$ which solves Problem \ref{prob 1} is an extremal pair for Problem \ref{prob 1}. The triple $(f_0,g_0,z_0)$ which solves Problem \ref{prob 1} is called an extremal triple for Problem \ref{prob 1}. It is quite straightforward to note that, $0< F(c)\leq 1$. Also, $F(c)$ is a non-increasing function in $(0,1)$ because for $0<c_1<c_2<1$, $A(c_2,1)\subseteq A(c_1,1)$ and thus the supremum of the smaller set $A(c_2,1)$ is less than the supremum of the larger set $A(c_1,1)$. Hence, if $F(c_0)=1$, then $F(c)=1$ for all $0<c\leq c_0$. Therefore, $F(c)=1$ for $0<c\leq \kappa$ and $F(c)<1$ for $\kappa<c<1$, where $\kappa$ is {\it Korenblum's constant}.

\begin{problem}\label{prob 2}%
Given $0<c<1$, 
\begin{enumerate}
\item[{\rm 1.}] Find%
\begin{equation} \label{eq 3.1}%
F_B(c):=\sup \limits_{f,g\in FG(c)} \left(\|f\|_2^2-\|g\|_2^2\right),
\end{equation} %
where
\begin{eqnarray*}%
FG(c):= \{ f,g \in \mathcal{A}^2(\mathbb{D}):\ \|f\|_2 \leq 1,\ \|g\|_2 \leq 1,\ |f(z)| \leq |g(z)|,\\
 \forall z: c\leq |z| <1 \}
\end{eqnarray*}%
is the set of all admissible pairs for $F_B(c)$.

\item[{\rm 2.}] Find an extremal pair of functions $f,g\in FG(c)$ such that $F_B(c)=\|f\|_2^2-\|g\|_2^2$.
\end{enumerate}
\end{problem} %
 We call the function $F_B(c)$ introduced in (\ref{eq 3.1}) as {\it Korenblum's function for Problem \ref{prob 2}}. Every pair of functions $(f_0,g_0)$ which solves Problem \ref{prob 2} is an {\it extremal pair for Problem \ref{prob 2}}. It is easy to see that $0\leq F_B(c)\leq 1$. Further, note that the extremal function $F_B(c)$ is a non-decreasing function. For $0<c_1<c_2<1$, $FG(c_1)\subseteq FG(c_2)$ and therefore the supremum of the smaller set $FG(c_1)$ is less than the supremum of larger set $FG(c_2)$. Therefore, $F_B(c)=0$ for $0<c\leq \kappa$ and $F(c)<1$ for $\kappa<c<1$, where $\kappa$ is {\it Korenblum's constant}.

\begin{problem}\label{prob 3}%
Given $0<c<1$, 
\begin{enumerate}
\item[{\rm 1.}] Find%
\begin{equation}\label{eq 3.1a}%
G(c)= \sup \limits_{f,g\in S^1} \inf \limits_{c\leq |z|<1} \left( |g(z)|^2-|f(z)|^2 \right) 
\end{equation}%

\item[{\rm 2.}] Find an extremal pair of functions $f,g\in S^1$ such that $G(c)=|f(z_1)|^2-|g(z_1)|^2$ where $z_1\in(c,1)$.
\end{enumerate}
\end{problem}%

The primary goal of this paper is to understand Korenblum's problem in the setting of extremal problems. The complete solution of either one of the Problems \ref{prob 1}, \ref{prob 2} or \ref{prob 3} will solve {\it Korenblum's Maximum Principle}. 


\section{Preliminary results on Bergman spaces}
\setcounter{equation}{0}  %
In this section, we discuss some well known facts from the theory of Bergman spaces. The proofs of these results can be found in \cite{DS}.
\begin{theorem}%
Let $1 <p<\infty$. If $f \in \mathcal{A}^p$, then the partial sums of the Taylor series converge in norm to $f$.
\end{theorem}%

\begin{lemma}\label{lem-norm}%
If $\displaystyle f(z)=\sum_{k=0}^{\infty} a_k z^k \in \mathcal{A}^2(\mathbb{D})$, then
\begin{equation}\label{eq 1.2}%
\|f\|_2= \left( \sum_{k=0}^{\infty} \frac{|a_k|^2}{k+1}\right)^{\frac{1}{2}}
\end{equation}%
\end{lemma}%

\begin{lemma}\label{lem-bounded}%
Let $f$ be a bounded analytic function on the unit disk $\mathbb{D}$. Then  $f \in \mathcal{A}^p(\mathbb{D})$.
\end{lemma}%

\begin{proof}%
Since $f$ is a bounded analytic function on the unit disk $\mathbb{D}$, then $|f(z)| \leq M$ for all $z \in \mathbb{D}$. Therefore, for some $0<\rho<1$, 
\begin{align*}%
\|f\|^p_{\mathcal{A}^p(\mathbb{D}_\rho)} =  \frac{1}{\pi}\int \limits_{\mathbb{D}_\rho} |f(z)| ^p\, dA(z) &=\frac{1}{\pi} \int \limits_{r=0}^{\rho} \int \limits_{\theta=0}^{2\pi} \left |f(re^{i\theta}) \right | ^p\ r\,dr\,d\theta \\
&\leq \frac{M}{\pi}\ \int_{r=0}^{\rho} \int_{\theta=0}^{2\pi} r\,dr\,d\theta \\
&=M \rho ^2.
\end{align*}%
As $\rho \rightarrow 1$, this gives $\| f \| ^p_{\mathcal{A}^p(\mathbb{D})} = M \ <\infty$. Therefore, $f \in \mathcal{A}^p(\mathbb{D})$. 
\end{proof}%

\begin{lemma}%
If $\displaystyle f(z)=\sum_{k=0}^{\infty} a_k z^k \in \mathcal{A}^2(\mathbb{D})$, then for $0<\rho<1$, $f_\rho(z)=f(\rho z) \in \mathcal{A}^2(\mathbb{D})$ and 
\begin{equation}\label{eq 1.3a}%
\|f_{\rho}\|_2= \left( \sum_{k=0}^{\infty} \frac{\rho^{2k} |a_k|^2}{(k+1)}\right)^{\frac{1}{2}}
\end{equation}%
Also, if $\|f\|_2<\infty$ and $f \not \equiv$ constant, then $\|f_{\rho}\|_2$ strictly increases from $|f(0)|$ to $\|f\|_2$ as $\rho$ runs from $0$ to $1$. Moreover,
\begin{equation}\label{eq 1.3b}%
\frac{d}{d\rho}\left( \|f_{\rho}\|_2^2 \right)= \sum_{k=1}^{\infty} \frac{2k \rho^{2k-1} |a_k|^2}{(k+1)}
\end{equation}%
\end{lemma}%

\begin{proof}%
This follows from the analyticity of $f_{\rho}$ in $\overline{\mathbb{D}}$.
\end{proof}%

\begin{lemma}%
Let $f\in \mathcal{A}^p(\mathbb{D})$. Define $f_{\rho}(z)=f(\rho z)$ for $0<\rho<1$. Then $\|f_{\rho}\|_p$ converges to $\|f\|_p$ as $\rho \to 1$.
\end{lemma}%


\section{Existence and properties of an extremal pair}
\setcounter{equation}{0}  %

In Problems  \ref{prob 1} and \ref{prob 2}, norms of the admissible functions are uniformly bounded by $1$. Therefore, by Montel's theorem, those functions form a normal family and hence we justify using the standard normal family arguments with careful technical details the existence of the extremal pair. We shall prove the existence result for Problem \ref{prob 1} in Proposition \ref{prop 1} and one can argue using similar arguments to prove the same for Problem \ref{prob 2}. 

\begin{proposition}\label{prop 1}%
An extremal pair of functions exists in Problem \ref{prob 1}.
\end{proposition}%

\begin{proof}%
Suppose $f_n$, $g_n$ is a sequence of functions in $S^1$ and $z_n\in \mathbb{C}$ such that $\left | f_n(z_n)/g_n(z_n) \right| \to F(c)$ where $c\leq |z_n|<1$ and $z_n\to z_0$ in $c\leq |z|<1$. Since $S^1$ is compact, there exists subsequences $f_{n_k}$, $g_{n_k}$ of $f_n$, $g_n$ respectively in $S^1$ such that $f_{n_k}\to f$, $g_{n_k}\to g$ uniformly on compact subsets of $\mathbb{D}$. Relabel $f_{n_k}$, $g_{n_k}$ as $f_n$, $g_n$ respectively. Since $\|f_n\|_2=\|g_n\|_2=1$ for all $n$, then $\|f\|_2=\|g\|_2=1$. Therefore, $f,g\in S^1$.\\
Case $1$: If $f_n(z)/g_n(z)$  does not have any zeroes, then they have removable singularites at the $z_n$'s. Therefore, they have analytic extensions in $c\leq |z|<1$ and hence $\left| f_n(z_n)/g_n(z_n) \right| \to \left| f(z_0)/g(z_0)\right|$ as $n\to\infty$. By uniqueness of limits, $\left| f(z_0)/g(z_0) \right|=F(c)$.\\
Case $2$: If $f_n(z)/g_n(z)$ have zeroes and say $g_n$ has a zero $\zeta$ of order $m$ where $\zeta\in \overline{A(c,1)}$. Then $g_n(z)=(z-\zeta)^m h(z),\ h(z) \neq 0$. Therefore, $f_n(z)$ should be of the form $f_n(z)= (z-\zeta)^{m+k} h_1(z),\ h_1(z)\neq 0$, otherwise $\sup_{c\leq |z|<1} \left | f_n(z)/g_n(z) \right |=+\infty$. Note that, the infimum supremum cannot be $+\infty$ since it is always less than or equal to $1$. Therefore, $f_n$ has a zero of order at least $m$ at $\zeta$. Define 
$\widetilde{g}_n(z)=g_n(z)/(z-\zeta)^m$ and $\widetilde{f}_n(z)=f_n(z)/(z-\zeta)^m$. Clearly, $\widetilde{f}_n$, $\widetilde{g}_n$ have removable singularities at $z_0$, therefore they are analytic. By Hurwitz theorem, since $f_n$ and $g_n$ have a zero of order $m$, then $f$ and $g$ also has a zero of order $m$ at $\zeta$. Therefore, $\widetilde{g}_n(z) \to g(z)/(z-\zeta)^m$ and $\widetilde{f}_n(z) \to f(z)/(z-\zeta)^m$ where $0<|z-\zeta|<1$. Therefore, $\left| f_n(z_n)/g_n(z_n) \right| =\left | \widetilde{f}_n(z_n)/\widetilde{g}_n(z_n) \right| \to \left| f(z_0)/g(z_0) \right|$. By uniqueness of limits, $\left| f(z_0)/g(z_0) \right| =F(c)$.
\end{proof}%

It is important to note that if $(f,g)$ is an extremal pair for $F(c)$, then the zeros of $g(z)$ can either lie inside $|z|<c$ or $A(c,1)$ or $|z|>1$. If all the zeros of $g(z)$ lie in $|z|>1$, then the quotient $f(z)/g(z)$ is analytic in $A(c,1)$ and hence by the Maximum Modulus Theorem is analytic in $\mathbb{D}$. In this case, $F(c)=1$. If zeros lie in $A(c,1)$, then the zeros of $f(z)$ and $g(z)$ have to cancel each other to make the quotient analytic in $A(c,1)$. Moreover, Lemma \ref{lem-zero} deals with the case when zeros are in $|z|<c$. 

\begin{lemma}\label{lem-zero}%
If $(f,g)$ is an extremal pair for $F(c)<1$, then $g(z)$ has a zero in $|z|<c$.
\end{lemma}%

\begin{proof}%
Following Hinkkanen, consider the function $\omega(z)=f(z)/g(z)$. If there is no zero in $|z|<c$, then $\omega(z)$ is analytic in $\mathbb{D}$ by analytic continuation. Therefore, $|\omega(z)|\leq F(c)$ for all $z$ in $\mathbb{D}$. This implies $|f(z)|\leq F(c) |g(z)|$ for all $z$ in $\mathbb{D}$. Then by Maximum Modulus Theorem, $f(z)=e^{i\beta}g(z)F(c)$. But $\|f\|_2=\|g\|_2F(c)<1$, contradicting the fact that $f, g \in S^1$.
\end{proof}%

\begin{lemma}\label{prop:2}%
Let $(f, g, z_0)$ be an extremal triple for $F(c)$. If $F(c)<1$, then $z_0\in \partial A(c,1)$ that is, either $|z_0|=c$ or $|z_0|=1$.
\end{lemma}%

\begin{proof}%
Let $(f, g)$ be an extremal pair, that is 
\begin{equation}\label{eq 2.2}%
\sup \limits_{c\leq |z|<1} \left| \frac{f(z)}{g(z)}\right |=F(c) 
\end{equation}%
then there is a sequence $z_n$ such that $z_n \to z_0$ such that
\begin{eqnarray*}%
\left| \frac{f(z_n)}{g(z_n)} \right| \to F(c)
\end{eqnarray*}%
as $n\to\infty$. (\ref{eq 2.2}) implies that $\left | f(z_n)/g(z_n) \right| \leq F(c)$ for all  $z_n$ such that  
$c\leq |z_n|<1$. This further implies $f/g$ is bounded, hence has no poles and therefore analytic. Suppose if possible, $\sup \left | f(z_n)/g(z_n) \right|=F(c)$ for some $z_0$ such that $c<|z_0|<1$. Then by Maximum Modulus Theorem,
\begin{eqnarray*}%
\frac{f(z_n)}{g(z_n)}=F(c) e^{i\beta} \Rightarrow f(z_n) =F(c) e^{i\beta} g(z_n).
\end{eqnarray*}%
But $\|f\|_2=F(c) \|g\|_2<1$, which is a contradiction to the fact that $f\in S^1$. So, this means $z_0 \in \partial A(c,1)$, either at $|z_0|=c$ or $|z_0|=1$. 
\end{proof}%

In fact, Lemma \ref{lem-triple} shows that the extremal is attained on the circle $|z|=c$.

\begin{lemma}\label{lem-triple}%
If $(f,g)$ is an extremal pair for $F(c)<1$, then there is a $z_0$ with $|z_0|=c$ such that $\left | f(z_0)/g(z_0) \right|=F(c)$.
\end{lemma}%

\begin{proof}%
Suppose that
\begin{equation}\label{lem-eqtriple-1.1}%
\sup \limits_{|z|=c} \left| \frac{f(z)}{g(z)}\right|<F(c)
\end{equation}%
By continuity, $\left| f(rz)/g(rz)\right|<F(c)$ or 
\begin{equation}\label{lem-eqtriple-1.2}%
|f(rz)|<F(c) |g(rz)|
\end{equation}%
for all $c\leq |z|<1$ and all $r<1$ sufficiently close to $1$. We have,
\begin{align*}%
1 = \int \limits_{|z|<1} |f(z)|^2\; dA(z)&= \int \limits_{|z|< r} |f(z)|^2\; dA(z)+\int \limits_{r\leq |z|<1} |f(z)|^2\; dA(z)\\
&=r^2 \int \limits_{|z|<1} |f(rz)|^2\; dA(z)+\int \limits_{r\leq|z|<1} |f(z)|^2\; dA(z)
\end{align*}%
Let $\widetilde{f}_r(z)=f(rz)$, $\widetilde{g}_r(z)=g(rz)$. Then, we have 
\begin{eqnarray*}%
\|\widetilde{f}_r\|_2^2 =r^{-2} \left( 1-\int \limits_{r\leq |z|<1} |f(z)|^2\; dA (z)\right),\quad \|\widetilde{g}_r\|_2^2 =r^{-2} \left( 1-\int \limits_{r\leq |z|<1} |g(z)|^2\; dA(z) \right).
\end{eqnarray*}
Consider functions
\begin{equation}\label{lem-eqtriple-1.4}%
f_r(z)=\frac{\widetilde{f}_r(z)}{\|\widetilde{f}_r\|_2},\quad g_r(z)=\frac{\widetilde{g}_r(z)}{\|\widetilde{g}_r\|_2}
\end{equation}%
Then $(f_r,g_r)$ is an admissible pair for $F(c)$ and for all $z$ such that $c\leq |z|<1$,
\begin{align*}%
\left | \frac{f_r(z)}{g_r(z)} \right| =\frac{\|\widetilde{g}_r\|_2}{\|\widetilde{f}_r\|_2} \left|\frac{\widetilde{f}_r(z)}{\widetilde{g}_r(z)} \right|
= \sqrt{\frac{1-\int \limits_{r\leq |z|<1} |g(z)|^2\; dA(z)}{1-\int \limits_{r\leq |z|<1} |f(z)|^2\; dA(z)}}\, \left| \frac{f(rz)}{g(rz)} \right|
< \left| \frac{f(rz)}{g(rz)} \right| <F(c).
\end{align*}%
The latter inequalities follows from (\ref{lem-eqtriple-1.2}). Now, $\sup_{c\leq |z|<1} \left| f_r(z)/g_r(z) \right|<F(c)$ contradicting the definition of $F(c)$. Thus our assumption (\ref{lem-eqtriple-1.1}) was wrong and the result follows.
\end{proof}%

 If $f$ and $g$ are polynomials of degree at most $n\geq 1$, then functions $f_r$ and $g_r$ defined by (\ref{lem-eqtriple-1.4}) are also polynomials of degree at most $n$. Therefore, our proof of Lemma \ref{lem-triple} gives the following:
\begin{corollary}%
If $(p,q)$ is an extremal pair of polynomials for $F_n(c)<1$, then there is $z_0$ with $|z_0|=c$ such that $ \left| p(z_0)/q(z_0) \right|=F_n(c)$.
\end{corollary}%


\section{Properties of Korenblum's function}%
\setcounter{equation}{0}  %

As discussed in Section $3$, it is straightforward to observe that $F(c)$ is non-increasing in the interval $(0,1)$. Since ($\kappa,1$) is the non-trivial range of $F(c)$, we are further interested in the following.

\begin{lemma}\label{lem-mon}%
Let $F:(0,1)\to(0,1)$ and $F_B:(0,1)\to(0,1)$ be as defined before.
\begin{enumerate}
\item[{\rm (1).}] $F(c)$ is a strictly decreasing function in $(\kappa,1)$.
\item[{\rm (2).}] If $(f, g)$ is an extremal pair for $F_B(c_0)$ and $c_0\in (\kappa,1)$, then 
\begin{eqnarray*}
\max \limits_{|z|=c_0} \left | \frac{f(z)}{g(z)} \right|=1.
\end{eqnarray*} 
\item[{\rm (3).}] $F_B(c)$ is a strictly increasing function in $(\kappa,1)$.
\end{enumerate}
\end{lemma}%

\begin{proof}%
(1). Suppose that $\kappa<c_1<c_2<1$ and $0<F(c_2)=F(c_1)<1$. Let $(f,g)$ be an extremal pair for $F(c)$. Consider $\omega(z)=f(z)/g(z)$. Then $\omega(z)$ is analytic in $A(c,1)$. Furthermore, $|\omega(z)|$ does not take its maximal value in $A(c,1)$ otherwise $\omega(z)$ must be constant by the Maximum Modulus Theorem. Therefore, $f(z)=F(c)e^{i\beta}g(z)$ with some $\beta \in \mathbb{R}$. The latter contradicts the assumption that $f, g \in S^1$ since $\|f\|_2=F(c)\|g\|_2<1$. In the case, 
$ \sup_{c_2\leq |z|<1} \left | f(z)/g(z) \right| <F(c_1)=F(c_2)$, we have a contradiction with the definition of $F(c_2)$. In the case, $\sup_{c_2\leq |z|<1} \left| f(z)/g(z) \right | =F(c_1)=F(c_2)$. Since
$(f,g)$ is extremal for $F(c_2)$ and $|\omega(z)|<F(c_2)$ for $|z|=c_2$, we obtain a contradiction with Lemma \ref{prop:2}, because $\sup_{0\leq \theta \leq 2\pi} \left| f(c_2 e^{i\theta})/g(c_2 e^{i\theta}) \right |< F(c_1)=F(c_2)$ and $(f,g)$ is an extremal pair for $F(c_2)$.\\

\noindent
(2). Suppose  $\max_{|z|=c_0} \left |f(z)/g(z) \right|<1$. For $0<\rho<1$, consider the functions $f_{\rho}(z)=f(\rho z)$ and $g_{\rho}(z)=g(\rho z)$. We extend the inequality to a larger domain such that
$\left|f_{\rho}(z)\right| \leq \left| g_{\rho}(z)\right|$ for all $z$ such that $c_0\leq |z|<1$.
Therefore, $f_{\rho},\ g_{\rho} \in FG(c_0)$. Also, we have
\begin{align*}
F(c_0) = \|f\|_2^2-\|g\|_2^2 &=\int \limits_{\mathbb{D}} |f(z)|^2\; dA(z)-\int \limits_{\mathbb{D}} |g(z)|^2\; dA(z) \\
&\leq \int \limits_{\mathbb{D}_{\rho}} |f(z)|^2\; dA(z)-\int \limits_{\mathbb{D}_{\rho}} |g(z)|^2\; dA(z) 
\end{align*}
Let $z=\rho \zeta$. Then
\begin{align*}
F(c_0) &\leq \int \limits_{\mathbb{D}} |f_{\rho}(\zeta)|^2 \rho^2\; dA(\zeta)-\int \limits_{\mathbb{D}} |g_{\rho}(\zeta)|^2 \rho^2\; dA(\zeta)= \rho^2(\|f_{\rho}\|_2^2-\|g_{\rho}\|_2^2)\\
&< \|f_{\rho}\|_2^2-\|g_{\rho}\|_2^2 \leq F(c_0)
\end{align*}
since $f$, $g \in F(c_0)$. Therefore, we get $F(c_0)<F(c_0)$, which is a contradiction.\\

\noindent
(3). Suppose $F(c_1)=F(c_2)$. Let $(f,g)$ be the pair of extremal functions for $F(c_1)$. Then it is also an extremal pair for $F(c_2)$. ($2$) implies $\max_{|z|=c_2} \left |f(z)/g(z) \right|$ is $1$. 
Then by the Maximum Modulus Theorem applied to $f(z)/g(z)$ in $c_1\leq |z|<1$, we have 
$\left| f(z)/g(z)\right|=1$ for all $z$ such that $c_1\leq |z| <1$ and therefore on $\mathbb{D}$ by analytic continuation. Hence $\|f\|_2^2-\|g\|_2^2=0$, which is a contradiction to the fact that $c_1,c_2\in (\kappa,1)$.
\end{proof}%

\begin{lemma}\label{lem-cont}%
Both $F_B(c)$ and $F(c)$ are continuous functions from $(0,1)$ to $(0,1)$. 
\end{lemma}%

\begin{proof}%
We shall prove the continuity of $F(c)$ here. The proof of continuity for $F_B(c)$ follows in a similar way.\\
\noindent
Case $1$: Let $c_n \nearrow c$ as $n\to \infty$. Since $F(c)$ is non-increasing in $(0,1)$ then $c_n\leq c$ implies $F(c_n)\geq F(c)$. Therefore, $\lim_{n\to\infty} F(c_n) \geq F(c)$. Next we claim that $\lim_{n\to \infty} F(c_n)\leq F(c)$. Let $(f,g)$ be extremal for $F(c)$. Then $\sup_{c\leq |z|<1} \left| f(z)/g(z)\right|=F(c)$. Suppose if possible, $\lim_{n\to\infty} F(c_n)=F_{-}(c)>F(c)$. For $\varepsilon=(F_{-}(c)-F(c))/2$, there exists $\delta>0$ such that 
\begin{eqnarray*}%
F(c-\delta)=\sup \limits_{c-\delta\leq |z|<1} \left| \frac{f(z)}{g(z)} \right|<F(c)+\varepsilon.
\end{eqnarray*}%
But $c-\delta<c_n<c$, which gives a contradiction to the fact that $F(c)$ is non-increasing. \\
\noindent
Case $2$: Let $c_n \searrow c$ as $n\to \infty$. Since $F(c)$ is non-increasing in $(0,1)$, then
$c_n\to c\Rightarrow F(c_n)\leq F(c) \Rightarrow \lim_{n\to\infty} F(c_n)\leq F(c)$. We next claim that $\lim_{n\to\infty} F(c_n)\geq F(c)$. Let $\delta(c_n)$ be defined as in (\ref{eq 6.5}) since the supremum of the larger set $c_n\leq |z|<1$ is greater than the supremum of the smaller set $c\leq |z|<1$. Then
\begin{equation}\label{eq 6.5}%
\sup \limits_{c_n \leq |z|<1} \left |\frac{f(z)}{g(z)} \right|\leq \frac{F(c)}{1-\delta(c_n)},
\end{equation}%
where $\delta(c_n)>0$ and $\delta(c_n)\to 0$ as $c_n\to c$. Define 
\begin{eqnarray*}
f_n(z) :=\frac{(1-\delta(c_n))f(z)}{|1-\delta(c_n)|}.
\end{eqnarray*}
Note that $\|f_n\|_2=1$, $\|g\|_2=1$. Thus, $(f_n, g)$ is admissible and $\left| f_n(z)/g(z)\right|\leq F(c_n)$. This implies,
\begin{eqnarray*}%
F(c)=\left|\frac{f(z)}{g(z)}\right|=\lim \limits_{n\to \infty} \left| \frac{f_n(z)}{g(z)}\right|\leq \lim \limits_{n\to \infty} F(c_n),
\end{eqnarray*}%
which completes the proof.
\end{proof}%

\begin{lemma}%
Both $F(c)$ and $F_B(c)$ are homeomorphisms from $[\kappa,1]$ to $[0,1]$.
\end{lemma}%

\begin{proof}%
$F(c)$ and $F_B(c)$ are both continuous and strictly monotonic by Lemma \ref{lem-cont} and \ref{lem-mon}. Then by the inverse function theorem, $F^{-1}(c)$ and $F_B^{-1}(c)$ both exist. Also, $F^{-1}(c)$ and $F_B^{-1}(c)$ are strictly monotonic. Thus, both $F^{-1}(c)$ and $F_B^{-1}(c)$ are continuous. Therefore, it is sufficient to prove 
\begin{enumerate}
\item[(i)] $\lim_{c\to 1^-} F(c)=0$ and $\lim_{c\to \kappa^+} F(c)=1 $.
\item[(ii)] $\lim_{c \rightarrow 1^-} F_B(c)=1$ and $\lim_{c \rightarrow \kappa^+} F_B(c)=0$.
\end{enumerate}
The second equalities in (i) and (ii) are true by the definition of functions $F(c)$ and $F_B(c)$ respectively. So, we are left to prove that $\lim_{c\to 1^-} F(c)=0$ and $\lim_{c \rightarrow 1^-} F_B(c)=1$. Fix $0<r<1$ sufficiently small, choose $N=N(n)$ such that for $n>N$, $\varepsilon(r,n)<1$. Take $c<1$ such that $\left | \frac{1+rz^n}{z^n+r} \right|<2$ for $c\leq |z|<1$. Consider $f(z)=1$, $g_n(z)=\frac{1}{r} \frac{z^n+r}{1+rz^n}/ \|\frac{1}{r} \frac{z^n+r}{1+rz^n} \|_2$. Clearly, $\|f\|_2=1$ and $ \|g_n\|_2= 1$. Also $g_n$ converges to $1$ as $n\to\infty$ uniformly on compact subsets of $\mathbb{D}$. Therefore, $(f, g_n)$ is admissible for $F(c)$. Then
\[
\left | \frac{f(z)}{g_n(z)}\right | =\left| r \frac{1+rz^n}{z^n+r} (1+\varepsilon(r,n))\right| \leq r \left | \frac{1+rz^n}{z^n+r}\right| |1+\varepsilon(r,n)| \leq 4r
\]
This means, $\omega_n(z) :=f(z)/g_n(z) \leq 4r$ for all $z$ such that $c\leq |z|<1$. This implies $|f(z)|\leq 4r |g_n(z)|$ for all $z$ in $\mathbb{D}$. Then by the Maximum Modulus Theorem, $f(z)= e^{i\beta} g(z) 4r$. But $\|f\|_2=4r \|g_n\|_2 =4r$, contradicting $f \in S^1$. Therefore, $0\leq |\omega_n(z)|\leq F(c)\leq 4r$. Thus $F(c) \to 0$ as $c\to 1$. To prove (ii), let $0<r<1,\ k>1$ and $n\geq 1$, where $n \in \mathbb{N}$. Consider $f_0(z)=1$ and $\displaystyle h_n(z)=k \frac{z^n-r^n}{1-r^n z^n}$. Then for every $n$, there is $c_n$, $0<c_n<1$ such that for all $z$ in $c_n\leq |z|<1$,
\[ 1=|f_0(z)|\leq k \left| \frac{z^n-r^n}{1-r^n z^n}\right| \]
Note that, since $r,|z|<1$, then 
\[ \lim_{n\to\infty} h_n(z)=k\lim_{n\to\infty} \frac{z^n-r^n}{1-r^n z^n}=k\frac{0-0}{1-0}=0 \]
that is, $h_n(z)\rightarrow 0$ as $n\rightarrow \infty$ and $\mathcal{A}^2(\mathbb{D})$ being a Banach space implies that $\|h_n\|_2 \rightarrow 0$ as $n\rightarrow\infty$. Also, $\|h_n\|_2\leq 1$ for $n\geq N$, where $N$ is sufficiently large depending on $k$. Therefore, $(f_0,h_n) \in FG(c_n)$ for all $n\geq N$. Note that, $\|f_0\|_2^2-\|h_n\|_2^2 \rightarrow 1$ as $n\rightarrow \infty$. Also,
$\|f_0\|_2^2-\|h_n\|_2^2 \leq F(c_n) \leq 1$. Taking the limit as $n\to \infty$, we obtain
\[ 1=\lim_{n\to \infty} \|f_0\|_2^2-\|h_n\|_2^2 \leq \lim_{n\to\infty} F(c_n) \leq \lim_{n\to \infty}1=1 \]
Therefore, $\lim \limits_{c \rightarrow 1^-} F_B(c)=1$.
\end{proof}%

Let us prove the following Theorem \ref{lem:prob2} which gives us an upper bound and a lower bound for $F_B(c)$ and $F(c)$ respectively (Figure \ref{fig:5} illustrates the lower bound for $F(c)$).

\begin{theorem}\label{lem:prob2}%
For $0<c<1$,
\begin{enumerate}[font=\normalfont]
\item[{\rm(1).}] $F_B(c) \leq c^2$.
\item[{\rm(2).}] $F(c)>\sqrt{1-c^2}$.
\end{enumerate}
\end{theorem}%

\begin{proof}%
(1). Lemma \ref{lem-cont} guarantees the continuity of functions $F(c)$ and $F_B(c)$. Further, $F_B(c)$ is non-decreasing which implies the derivative of $F_B(c)$ which is $F_B^{\prime}(c)$ exists almost everywhere. Thus, we will consider points $c$ such that $F_B^{\prime}(c)$ exists. Let $(f,g)$ be extremal for $F(c)$ and let $\rho<1$ be sufficiently close to $1$. Consider $f_{\rho}(z)=f(\rho z)$ and $g_{\rho}(z)=g(\rho z)$. Then
\begin{align*}%
F_B(c) &= \|f\|_2^2-\|g\|_2^2 \\
&=\int \limits_{\mathbb{D}} |f(z)|^2\; dA(z)-\int \limits_{\mathbb{D}} |g(z)|^2\; dA(z) \\
&<\int \limits_{\mathbb{D}_{\rho}} |f(z)|^2\; dA(z)-\int \limits_{\mathbb{D}_{\rho}} |g(z)|^2\; dA(z) \\
&= \int \limits_{\mathbb{D}} |f_{\rho}(\zeta)|^2 \rho^2\; dA(\zeta)-\int \limits_{\mathbb{D}} |g_{\rho}(\zeta)|^2 \rho^2\; dA(\zeta)\\
&= \rho^2(\|f_{\rho}\|_2^2-\|g_{\rho}\|_2^2) \leq \rho^2 F_B\left(c/\rho\right)\\
\end{align*}%
The last inequality follows from the fact that $|f_{\rho}(z)| \leq |g_{\rho}(z)|$ for all $z$ such that  $c/\rho\leq |z|<1$. Therefore, we obtain $F_B(c) \leq \rho^2 F_B\left(c/\rho\right)$. This further implies
 \begin{eqnarray*}%
 F_B^{\prime}(c)=\lim_{\rho \rightarrow 1^{-}} \frac{F_B \left(c/\rho\right)-F_B(c)}{c/\rho-c} \geq \lim_{\rho \rightarrow 1^{-}} \frac{F_B \left(c/\rho\right)-\rho^2 F_B\left(c/\rho\right)}{c(1-\rho)/\rho}=\frac{2 F_B(c)}{c}.
\end{eqnarray*}%
 Thus,
 \begin{equation}\label{eq 3.6d}%
 \frac{F_B^{\prime}(c)}{F_B(c)} \geq \frac{2}{c} 
 \end{equation}%
 Integrating (\ref{eq 3.6d}) from $c$ to $1$, we obtain $\log F_B(1)- \log F_B(c) \geq -2\log c$. Equivalently $F_B(c)\leq c^2$.\\

\noindent
(2). Let $(f,g)$ be an extremal pair for $F(c)$. Then $|f(z)|\leq F(c) |g(z)|$ for all $z$ such that $c\leq |z|<1$. Therefore, the pair of functions $(f,F(c)g)$ is admissible for Problem \ref{prob 2}. Thus by Lemma \ref{lem:prob2} ($1$), 
\begin{eqnarray*}%
\int \limits_{|z|<1} |f(z)|^2\; dA(z)-\int \limits_{|z|<1} F^2(c) |g(z)|^2\; dA(z)\leq F_B(c)<c^2, 
\end{eqnarray*}%
for all $0<c<1$. Since $f$, $g \in S^1$, then $1-F^2(c)<c^2$ for $0<c<1$. Thus the result follows. 
\end{proof}%


\section{Some special cases}%
\setcounter{equation}{0} %

We shall now introduce problems related to polynomials, Blaschke products and bounded functions. 

\subsection{Polynomials of degree n}%
For $n\geq 1$, consider the class $\mathcal{P}_n$ of polynomials of degree at most $n$. 

\begin{problem}%
Given $0<c<1$ and $n\geq 1$,
\begin{enumerate}
\item[{\rm 1.}] Find
\begin{equation}\label{eq 3.5}%
F_n(c)= \inf \limits_{p,q\in S^1 \cap \mathcal{P}_n} \sup \limits_{c\leq |z|<1} \left |\frac{p(z)}{q(z)} \right |
\end{equation}%

\item[{\rm 2.}] Find an extremal pair of polynomials $p,q\in S^1$ of degree at most $n$ such that $F(c)=\left|p(z_2)/q(z_2)\right|$ for some $z_2 \in  (c,1)$.
\end{enumerate}
\end{problem}%

Note that, $0<F_n(c)\leq 1$. Further, $\kappa_n$ is Korenblum's constant for polynomials of degree at most $n\geq 1$ such that $F_n(c)=1$ for $0<c\leq \kappa_n$ and $ F_n(c)<1$ for $\kappa_n<c<1$. Note that, since $\mathcal{P}_n \subseteq \mathcal{P}_{n+1}$, it follows that
\begin{eqnarray*}%
0<\kappa \leq \kappa_{n+1} \leq \kappa_n \leq \ldots \leq \kappa_2 \leq \kappa_1 < 1.
\end{eqnarray*}%

\begin{lemma}\label{lem-limpol}
 $\kappa =\lim_{n\to\infty} \kappa_n$.
 \end{lemma}%
 
 \begin{proof}%
 Clearly, $\kappa\leq \kappa_n$ for all $n \in \mathbb{N}$. Consider a sequence $c_m \rightarrow \kappa$ such that $c_m>c_{m+1}$. Let $(f_m,g_m)$ be an extremal pair for $F(c_m)$. Then
 \begin{eqnarray*}%
 F(c_m)=\sup \limits_{c_m\leq |z|<1} \left | \frac{f_m(z)}{g_m(z)}\right|\leq 1.
\end{eqnarray*}%
 Therefore,
 \begin{equation}\label{eq 3.9a}%
|f_m(z)|\leq |g_m(z)|,\ \forall z,\ c_m\leq |z|<1. 
\end{equation}%
Note that, to achieve a strict inequality in (\ref{eq 3.9a}), consider
\begin{equation}\label{eq 3.9f}%
|f_m(z)|< |g_m(z)|,\ \forall z,\ c_m+\delta_m\leq |z|<1.
\end{equation}%
Consider the $n$th partial sums of Taylor series of $f_m(z)$ and $g_m(z)$ as
\begin{eqnarray*}%
S_n(f_m)(z)=\sum_{k=0}^n a_{k,m} z^k,\; S_n(g_m)(z)=\sum_{k=0}^n b_{k,m} z^k.
\end{eqnarray*}%
For $\varepsilon_m>0$ and sufficiently small, we obtain
\begin{equation}\label{eq 3.9b}%
|f_m(z)-S_n(f_m)(z)|<\frac{\varepsilon_m}{2},\ \forall z,\ |z|\leq 1-\delta_m,\ \forall n\geq N_1(m),
\end{equation}%
and 
\begin{equation}\label{eq 3.9c}%
 |g_m(z)-S_n(g_m)(z)|<\frac{\varepsilon_m}{2},\ \forall z,\ |z|\leq 1-\delta_m,\ \forall n\geq N_2(m),
\end{equation}%
where $\delta_m\rightarrow 0$ as $m \rightarrow \infty$. Choose $N(m)=\max\{N_1(m),N_2(m)\}$ where (\ref{eq 3.9b}), (\ref{eq 3.9c}) hold true. Then
\begin{align*}%
\left | S_n(f_m)(z)\right| &= |f_m(z)-S_n(f_m)(z)-f_m(z)|\leq |f_m(z)-S_n(f_m)(z)|+|f_m(z)|\\
&<\frac{\varepsilon_m}{2}+|f_m(z)|,
\end{align*}%
and 
\begin{align*}%
\left | S_n(g_m)(z)\right| &= |g_m(z)-S_n(g_m)(z)-g_m(z)|\geq |g_m(z)|-|g_m(z)-S_n(g_m)(z)|\\
&>|g_m(z)|-\frac{\varepsilon_m}{2}.
\end{align*}%
Combining last two inequalities, we obtain
\begin{eqnarray*}%
|S_n(f_m)(z)|-|S_n(g_m)(z)|< \frac{\varepsilon_m}{2}+|f_m(z)|- |g_m(z)|+\frac{\varepsilon_m}{2}<0,
\end{eqnarray*}%
for all $z$ such that $c_m+\delta_m\leq |z|\leq 1-\delta_m$ using (\ref{eq 3.9f}). Therefore,
\begin{eqnarray*}%
|S_n(f_m)(z)|\leq |S_n(g_m)(z)|,\ \forall z,\ c_m+\delta_m\leq |z|\leq 1-\delta_m,
\end{eqnarray*}%
which therefore implies,
\begin{eqnarray*}%
|S_n (f_m)((1-\delta_m)z)|\leq |S_n (g_m)((1-\delta_m)z)|,\ \forall z,\ \frac{c_m+\delta_m}{1-\delta_m} \leq |z|\leq 1.
\end{eqnarray*}%
For $n\geq N$, $\|S_n(f_m)\|_2=\|S_n(g_m)\|_2=1$. We follow the same steps of Lemma \ref{lem-triple} by replacing $f$ and $g$ by $S_n(f_m)$ and $S_n(g_m)$ respectively. We choose $r=1-\delta_m$ where $c<r<1$. Also $\|(\widetilde{S}_n(f_m))_r\|_2=\|(\widetilde{S}_n(g_m))_r\|_2=1$ and hence $((\widetilde{S}_n(f_m))_r, (\widetilde{S}_n(g_m))_r)$ is admissible and are polynomials. Therefore, $F_n((c_m+\delta_m)/(1-\delta_m))<1$. Also, $(c_m+\delta_m)/(1-\delta_m) \to \kappa$ as $m\to\infty$, that is, $\kappa_n\to\kappa$ as $n\to\infty$.
 \end{proof}%


\subsection{Polynomials of Degree 1}\label{subsec:polF(c)1}%
\setcounter{equation}{0}  %

We recall from Section $1$ that for the wider class of functions, Hayman's example provides an upper bound for $\kappa$, that is $\kappa<1/\sqrt{2}$. In this section, we claim that $1/\sqrt{2}$ is the sharp constant for the class of polynomials of degree $1$. Let $(p,q,z_0)$ be an extremal triple for $F_1(c)$ with $0<c<1$ such that $F_1(c)<1$. Without loss of generality, we may assume that
$p(z)=\alpha+\beta e^{it}z$, $q(z)=\gamma+\delta z$ where $\alpha\geq0$, $\beta\geq 0$, $\gamma\geq 0$, $\delta\geq 0$. Then the quotient is a M\"{o}bius map:
\begin{equation}\label{deg1-1.2}%
\varphi_t(z)=\frac{\alpha+\beta e^{it}z}{\gamma+\delta z}.
\end{equation}%
Since $\max_{c\leq |z|\leq 1} |\varphi_t(z)|=|\varphi_t(z_0)|$, where $|z_0|=c$, it follows that the pole of $\varphi_t$ should be in the disk $|z|<c$. Therefore, we have $0\leq \gamma/\delta<c$.
Since $\varphi_t$ is M\"{o}bius with pole at $z=-\gamma/\delta$. It maps $|z|=c$ onto a circle centered at the point $\varphi_t (-\delta c^2/\gamma)$ which is 
\begin{equation}\label{deg1-1.4}%
\varphi_t (-\delta c^2/\gamma)=\frac{\alpha-\frac{\beta \delta c^2}{\gamma}e^{it}}{\gamma-\frac{\delta^2 c^2}{\gamma}}=\frac{\alpha \gamma -\beta \delta c^2 e^{it}}{\gamma^2-\delta^2c^2}.
\end{equation}%
The radius $R$ of the image circle $\varphi(|z|=c)$ can be found as follows:
\begin{equation}\label{deg1-1.5}%
 R=\left | \varphi_t(c)-\varphi_t(-\delta c^2/\gamma)\right|=\left | \frac{\alpha+\beta c e^{it}}{\gamma +\delta c}-\frac{\alpha \gamma -\beta \delta c^2 e^{it}}{\gamma^2-\delta^2c^2} \right|=\frac{c}{\delta^2c^2-\gamma^2} \left| \alpha \delta-\beta \gamma e^{it}\right|.
\end{equation}%
It follows from (\ref{deg1-1.4}) and (\ref{deg1-1.5}) that the furthest point of the circle $\varphi(|z|=c)$ has modulus
\begin{equation}\label{deg1-1.6}%
\left| \varphi_t ( -\delta c^2/\gamma)\right|+R=\frac{1}{\delta^2c^2-\gamma^2} \left( \left |\alpha \gamma-\beta \delta c^2 e^{it} \right|+c\left| \alpha \delta -\beta \gamma e^{it} \right| \right)
\end{equation}%
Since all the parameters $\alpha$, $\beta$, $\gamma$, $\delta$ and $c$ are non negative, it follows from (\ref{deg1-1.6}) that, for fixed $\alpha$, $\beta$, $\gamma$, $\delta$ and $c$, the distance (\ref{deg1-1.6}) is smallest when $t=0$. In the latter case we have:
\begin{equation}\label{deg1-1.7}%
\varphi_0(z)=\frac{\alpha+\beta z}{\gamma+\delta z}
\end{equation}%
Since the substitution $z \mapsto e^{i\theta}z$ does not change the maximum of $|\varphi_0(z)|$ over the circle $|z|=r$ and since $\beta=\sqrt{2(1-\alpha^2)}$, $\delta=\sqrt{2(1-\gamma^2}$, we can change $\varphi_0(z)$  to the following form $\varphi(z)$, which will be more convenient for our purposes.
\begin{eqnarray*}%
\varphi(z)=\frac{\sqrt{1+2b^2}}{\sqrt{1+2a^2}}\, \frac{z-a}{z-b},
\end{eqnarray*}%
where
\begin{eqnarray*}%
a=\frac{\alpha}{\beta}=\frac{\alpha}{\sqrt{2(1-\alpha^2)}},\quad b=\frac{\gamma}{\delta}=\frac{\gamma}{\sqrt{2(1-\gamma^2)}}.
\end{eqnarray*}%
We note that the maximum of $\left| \varphi(ce^{it})\right|$ for $0\leq t\leq 2\pi$ may occur when $t=0$ or when $t=\pi$ depending on parameters $a$ and $b$. We consider the following cases:

\begin{figure}%
\centering %
\hspace{-4.3cm} %
\begin{minipage}{0.65\linewidth}%
\begin{tikzpicture}%
    [inner sep=1mm,
     place/.style={circle,draw=blue!50,
     fill=blue!20,thick},
		 scale=2.0]%
\draw [black, very thick, fill=gray!30!white](1,0) circle(1.0);%
\draw [black, very thick, fill=white](1,0) circle(0.55);%
\draw [black,dashed] (-0.1,0)--(2.5,0);%
\draw [<-, >=stealth, red, very thick] (1.2,0) to (1.48,0);%
\draw [->,>=stealth,black, very thick] (2.7,0) to (3.2,0);%
   
\node at (2.9,0) [above] {\small{$\varphi$}};%
\node at (1,0) [] {\tiny{$\bullet$}}; %
\node at (1.0,-0.01) [below] {\small{$0$}}; %
\node at (1.2,0) [] {\tiny{$\bullet$}};%
\node at (1.2,-0.01) [below] {\small{$a$}}; %
\node at (1.48,0) [] {\tiny{$\bullet$}}; %
\node at (1.48,-0.01) [below] {\small{$b$}}; %
\node at (1.55,0) [] {\tiny{$\bullet$}}; %
\node at (1.55,-0.1) [right] {\small{$c$}}; %
\node at (2,0) [] {\tiny{$\bullet$}}; %
\node at (2,-0.1) [right] {\small{$1$}}; %

\draw [black, very thick, fill=gray!30!white](4.9,0) circle(1.0);%
\draw [black, very thick, fill=white](4.7,0) circle(0.3);%
\draw [black,dashed] (3.4,0)--(6.0,0);%
\draw [->, >=stealth, red, very thick] (3.4,0) to (3.8,0);%
\draw [blue, very thick] (5.9,0) circle (2pt);%

\node at (3.8,0) [] {\tiny{$\bullet$}}; %
\node at (3.8,-0.01) [below] {\small{$0$}}; %
\node at (3.9,0) [] {\tiny{$\bullet$}}; %
\node at (3.9,-0.1) [right] {\small{$\varphi(-c)$}}; %
\node at (5.9,0) [] {\tiny{$\bullet$}}; %
\node at (5.9,-0.1) [left] {\small{$\varphi(c)$}}; %
\end{tikzpicture}%
\end{minipage}%
\caption{Case (1).}%
\label{fig:1}%
\end{figure}%

\noindent ($1$). If $0\leq a<b<c<1$, then $\max_{0\leq t\leq 2\pi} |\varphi(ce^{it})|=\varphi(c)$; see Figure \ref{fig:1}. We have,
\begin{eqnarray*}%
\varphi(c)=\frac{\sqrt{1+2b^2}}{\sqrt{1+2a^2}}\, \frac{c-a}{c-b}.
\end{eqnarray*}%
We minimize $\varphi(c)$ with respect to variable $a$ on the interval $0\leq a<b$. Differentiating, we find
\begin{eqnarray*}%
\frac{\partial}{\partial a} \left( \frac{c-a}{\sqrt{1+2a^2}}\right)=-\frac{1+2ac}{(1+2a^2)^{3/2}}<0.
\end{eqnarray*}%
Therefore, the minimal value of $\varphi(c)$ will occur when $a=b$, which gives $\varphi(c)=1$. Thus, the quotient $p(z)/q(z)$ does not give a non-trivial value for $F_1(c)$ in this case.\\
\noindent ($2$). If $0<b<a<c<1$, then $\max_{0\leq t\leq 2\pi} |\varphi(ce^{it})|=\varphi(-c)$; see Figure \ref{fig:2}.\\
In this case, we have 
\begin{equation}\label{deg1-1.9}%
\varphi(-c)=\frac{\sqrt{1+2b^2}}{\sqrt{1+2a^2}}\, \frac{c+a}{c+b}.
\end{equation}%
Differentiating with respect to $a$, we find
\begin{eqnarray*}%
\frac{\partial}{\partial a} \left( \frac{c+a}{\sqrt{1+2a^2}}\right)=\frac{1-2ac}{(1+2a^2)^{3/2}}.
\end{eqnarray*}%
The latter gives one critical point $a=1/(2c)$. This point will be in the required interval if $1/(2c)<c$ or $c>1/\sqrt{2}$. Calculating the values of $\varphi(-c)$ for $a=b$, $a=c$ and $a=1/(2c)$, we find\\
\noindent (i) If $a=b$, then $\varphi(-c)=1$.

\begin{figure}%
\centering %
\hspace{-4.3cm} %
\begin{minipage}{0.70\linewidth}%
\begin{tikzpicture}%
    [inner sep=1mm,
     place/.style={circle,draw=blue!50,
     fill=blue!20,thick},
     scale=2.0]%
\draw [black, very thick, fill=gray!30!white](1,0) circle(1.0);%
\draw [black, very thick, fill=white](1,0) circle(0.55);%
\draw [black,dashed] (-0.1,0)--(2.5,0);%
\draw [->, >=stealth, red, very thick] (1.2,0) to (1.48,0);%
\draw [->,>=stealth,black, thick] (2.7,0) to (3.2,0);%
   
\node at (2.9,0) [above] {\small{$\varphi$}};%
\node at (1,0) [] {\tiny{$\bullet$}}; %
\node at (1.0,-0.01) [below] {\small{$0$}}; %
\node at (1.2,0) [] {\tiny{$\bullet$}};%
\node at (1.2,-0.01) [below] {\small{$b$}}; %
\node at (1.48,0) [] {\tiny{$\bullet$}}; %
\node at (1.48,-0.01) [below] {\small{$a$}}; %
\node at (1.55,0) [] {\tiny{$\bullet$}}; %
\node at (1.55,-0.1) [right] {\small{$c$}}; %
\node at (2,0) [] {\tiny{$\bullet$}}; %
\node at (2,-0.1) [right] {\small{$1$}}; %

\draw [black, very thick, fill=gray!30!white](4.9,0) circle(1.0);%
\draw [black, very thick, fill=white](4.7,0) circle(0.3);%
\draw [black,dashed] (3.4,0)--(6.0,0);%
\draw [->, >=stealth, red, very thick] (3.4,0) to (3.8,0);%
\draw [blue, very thick] (5.9,0) circle (2pt);%

\node at (3.8,0) [] {\tiny{$\bullet$}}; %
\node at (3.8,-0.01) [below] {\small{$0$}}; %
\node at (3.9,0) [] {\tiny{$\bullet$}}; %
\node at (3.9,-0.1) [right] {\small{$\varphi(c)$}}; %
\node at (5.9,0) [] {\tiny{$\bullet$}}; %
\node at (5.9,-0.1) [left] {\small{$\varphi(-c)$}}; %
\end{tikzpicture}%
\end{minipage}%

\caption{Case (2).}%
\label{fig:2}%
\end{figure}%

\noindent
(ii) If $a=c$, then
\begin{eqnarray*}%
\varphi(-c)=\frac{2c}{\sqrt{1+2c^2}}\, \frac{\sqrt{1+2b^2}}{c+b}.
\end{eqnarray*}%
Differentiating with respect to $b$, we obtain
\begin{eqnarray*}%
\frac{\partial}{\partial b} \left(\frac{\sqrt{1+2b^2}}{c+b} \right)=\frac{2bc-1}{\sqrt{1+2b^2} (c+b)^2}.
\end{eqnarray*}%
The only critical point in this case is $b=1/(2c)$. Since $b<c$, we have $1/(2c)<c$ or $c>1/\sqrt{2}$. Calculating values of $\varphi(-c)$ for $b=0$, $b=c$ and $b=1/(2c)$, we find
\begin{itemize}%
\item If $b=0$, then $\varphi(-c)=2/\sqrt{1+2c^2}>1$.
\item If $b=c$, then $\varphi(-c)=1$.
\item If $b=1/(2c)$, then
\begin{equation}\label{deg1-1.9a}%
\varphi(-c)=\frac{2\sqrt{2}c}{1+2c^2}<1\quad {\rm for}\quad \frac{1}{\sqrt{2}}<c<1.
\end{equation}%
\end{itemize}%
\noindent
(iii) If $a=1/(2c)$, then
\begin{equation}\label{deg1-1.9b}%
\varphi(-c)=\frac{\sqrt{1+2c^2}}{\sqrt{2}}\, \frac{\sqrt{1+2b^2}}{c+b}.
\end{equation}%
Thus, $b=1/(2c)$ is a critical point as in case (ii). Calculating values at $b=0$, $b=c$ and $b=1/(2c)$ we find
\begin{itemize}%
	\item If $b=0$, then $\varphi(-c)=\sqrt{1+2c^2}/(\sqrt{2}c)> \sqrt{3/2}>1$.
	\item If $b=c$, then $\varphi(-c)=(1+2c^2)/(2\sqrt{2}c)\geq (1+2(1/\sqrt{2})^2)/(2\sqrt{2}(1/\sqrt{2})=1$.
	\item If $b=1/(2c)$, then $c>1/\sqrt{2}$ and
	\begin{eqnarray*}
	\varphi(-c)=\frac{\sqrt{1+2c^2}}{\sqrt{2}}\, \frac{\sqrt{1+2\left(\frac{1}{4c^2}\right)}}{c+\frac{1}{2c}}=1.
	\end{eqnarray*}
	Thus, in case (iii) $\varphi(-c)$ does not give a non-trivial bound.
\end{itemize}%

\begin{figure}%
\centering %
\hspace{-4.3cm} %
\begin{minipage}{0.70\linewidth}%
\begin{tikzpicture}%
    [inner sep=1mm,
     place/.style={circle,draw=blue!50,
     fill=blue!20,thick},
     scale=2.0]%
\draw [black, very thick, fill=gray!30!white](1,0) circle(1.0);%
\draw [black, very thick, fill=white](1,0) circle(0.55);%
\draw [black,dashed] (-0.1,0)--(2.5,0);%
\draw [->, >=stealth, red, very thick] (1.2,0) to (1.65,0);%
\draw [->,>=stealth,black, thick] (2.7,0) to (3.2,0);%
   
\node at (2.9,0) [above] {\small{$\varphi$}};%
\node at (1,0) [] {\tiny{$\bullet$}}; %
\node at (1.0,-0.01) [below] {\small{$0$}}; %
\node at (1.2,0) [] {\tiny{$\bullet$}};%
\node at (1.2,-0.01) [below] {\small{$b$}}; %
\node at (1.55,0) [] {\tiny{$\bullet$}}; %
\node at (1.55,-0.1) [left] {\small{$c$}}; %
\node at (1.65,0) [] {\tiny{$\bullet$}}; %
\node at (1.65,-0.01) [below] {\small{$a$}}; %
\node at (1.8,0) [] {\tiny{$\bullet$}}; %
\node at (1.8,-0.01) [below] {\small{$\frac{c^2}{b}$}}; %
\node at (2,0) [] {\tiny{$\bullet$}}; %
\node at (2,-0.1) [right] {\small{$1$}}; %

\draw [black, very thick, fill=gray!30!white](4.9,0) circle(1.0);%
\draw [black, very thick, fill=white](5.11,0) circle(0.3);%
\draw [black,dashed] (3.4,0)--(6.0,0);%
\draw [->, >=stealth, red, very thick] (3.4,0) to (4.1,0);%
\draw [blue, very thick] (5.9,0) circle (2pt);%

\node at (3.9,0) [] {\tiny{$\bullet$}}; %
\node at (3.9,-0.1) [left] {\small{$\varphi(c)$}}; %
\node at (4.1,0) [] {\tiny{$\bullet$}}; %
\node at (4.1,-0.01) [below] {\small{$0$}}; %
\node at (4.5,0) [] {\tiny{$\bullet$}}; %
\node at (4.5,-0.01) [below] {\small{$\varphi (c^2/b)$}}; %
\node at (5.9,0) [] {\tiny{$\bullet$}}; %
\node at (5.9,-0.1) [left] {\small{$\varphi(-c)$}}; %
\end{tikzpicture}%
\end{minipage}%
\caption{Case (3).}%
\label{fig:3}%
\end{figure}%

\noindent ($3$). If $0<b<c<a<c^2/b$, then $\max_{0\leq t\leq 2\pi} |\varphi(ce^{it})|=\varphi(-c)$; see Figure \ref{fig:3}. In this case $\varphi(-c)$ is given by (\ref{deg1-1.9}) and the critical point (when $\varphi(-c)$ is considered as a function of $a$) is $a=1/(2c)$, $c<1/\sqrt{2}$. We again consider cases:\\
\noindent
(i) If $a=c$, then \begin{eqnarray*}
\varphi(-c)=\frac{2c}{\sqrt{1+2c^2}}\, \frac{1+2b^2}{c+b}.
\end{eqnarray*}
 As before, we have one critical point $b=1/(2c)$ if $c>1/\sqrt{2}$.
\begin{itemize}%
	\item For $b=0$, $\varphi(-c)=2/\sqrt{1+2c^2}>1$.
	\item For $b=1/(2c)$, $\varphi(-c)=(2\sqrt{2}c)/(1+2c^2)<1$ for $1/\sqrt{2}\leq c<1$.
	\item For $b=c$, $\varphi(-c)=1$.
\end{itemize}%
 Thus, $\varphi(-c)$ does not give a non-trivial solution.\\
\noindent 
(ii) If $a=1/(2c)$, then $\varphi(-c)$ is given by (\ref{deg1-1.9b}). As before, the only critical point is $b=1/(2c)$. But this is in the required interval if $b=1/(2c)$ or $c>1/\sqrt{2}$ and hence we cannot consider $b=1/(2c)$ as a critical point. We consider the cases $b=0$ and $b=c$ as done in Case ($2$), part (iii). Thus, $\varphi(-c)$ does not give a non-trivial solution.\\
\noindent
(iii) For $a=c^2/b$, 
\begin{equation}\label{deg1-1.10}%
\varphi(-c)=\frac{c\sqrt{1+2b^2}}{\sqrt{b^2+2c^4}}.
\end{equation}%
 Differentiating, we find
\begin{equation}\label{deg1-1.11}%
\frac{\partial}{\partial b} \sqrt{\frac{1+2b^2}{b^2+2c^4}}=\frac{1}{2} \sqrt{\frac{b^2+2c^4}{1+2b^2}} \frac{2b(4c^4-1)}{(b^2+2c^4)^2}.
\end{equation}%
If $0<c<1/\sqrt{2}$, then the derivative in (\ref{deg1-1.11}) is negative and therefore the minimum occurs at $b=c$ in this case: If $b=c$, $\varphi(-c)=1$ and there is no non-trivial solution. If $c>1/\sqrt{2}$, then the derivative in (\ref{deg1-1.11}) is positive and therefore the minimum occurs at $b=0$. If $b=0$, then 
\begin{equation}\label{deg1-1.11a}%
\varphi(-c)=\frac{1}{\sqrt{2}c},\quad \frac{1}{\sqrt{2}}<c\leq 1.
\end{equation}%

\begin{figure}%
\centering %
\hspace{-4.3cm} %
\begin{minipage}{0.70\linewidth}%
\begin{tikzpicture}%
    [inner sep=1mm,
     place/.style={circle,draw=blue!50,
     fill=blue!20,thick},
     scale=2.0]%
		
\draw [black, very thick, fill=gray!30!white](1,0) circle(1.0);%
\draw [black, very thick, fill=white](1,0) circle(0.55);%
\draw [black,dashed] (-0.1,0)--(2.5,0);%
\draw [->, >=stealth, red, very thick] (1.2,0) to (1.8,0);%
\draw [->,>=stealth,black, thick] (2.7,0) to (3.2,0);%
   
\node at (2.9,0) [above] {\small{$\varphi$}};%
\node at (1,0) [] {\tiny{$\bullet$}}; %
\node at (1.0,-0.01) [below] {\small{$0$}}; %
\node at (1.2,0) [] {\tiny{$\bullet$}};%
\node at (1.2,-0.01) [below] {\small{$b$}}; %
\node at (1.55,0) [] {\tiny{$\bullet$}}; %
\node at (1.55,-0.1) [left] {\small{$c$}}; %
\node at (1.65,0) [] {\tiny{$\bullet$}}; %
\node at (1.65,-0.01) [below] {\small{$\frac{c^2}{b}$}}; %
\node at (1.8,0) [] {\tiny{$\bullet$}}; %
\node at (1.8,-0.01) [below] {\small{$a$}}; %
\node at (2,0) [] {\tiny{$\bullet$}}; %
\node at (2,-0.1) [right] {\small{$1$}}; %

\draw [black, very thick, fill=gray!30!white](4.9,0) circle(1.0);%
\draw [black, very thick, fill=white](5.11,0) circle(0.3);%
\draw [black,dashed] (3.4,0)--(6.0,0);%
\draw [->, >=stealth, red, very thick] (3.4,0) to (4.5,0);%
\draw [blue, very thick] (3.9,0) circle (2pt);%

\node at (3.9,0) [] {\tiny{$\bullet$}}; %
\node at (3.9,-0.1) [left] {\small{$\varphi(c)$}}; %
\node at (4.1,0) [] {\tiny{$\bullet$}}; %
\node at (4.18,-0.02) [below] {\small{$\varphi (c^2/b)$}}; %
\node at (4.5,0) [] {\tiny{$\bullet$}}; %
\node at (4.5,-0.01) [below] {\small{$0$}}; %
\node at (5.9,0) [] {\tiny{$\bullet$}}; %
\node at (5.9,-0.1) [left] {\small{$\varphi(-c)$}}; %
\end{tikzpicture}%
\end{minipage}%
\caption{Case (4).}%
\label{fig:4}%
\end{figure}%

\noindent ($4$). If $0<b<c<c^2/b<a$, then $\max_{0\leq t\leq 2\pi} |\varphi(ce^{it})|=-\varphi(c)$; see Figure \ref{fig:4}. In this case $-\varphi(c)$ is given by the formula:
\begin{eqnarray}\label{deg1-1.8}%
-\varphi(c)=\frac{\sqrt{1+2b^2}}{\sqrt{1+2a^2}}\, \frac{a-c}{c-b}.
\end{eqnarray}%
Differentiating, we obtain
\begin{eqnarray*}%
\frac{\partial}{\partial a} \left( \frac{a-c}{\sqrt{1+2a^2}}\right)=\frac{1+2ac}{(1+2a^2)^{3/2}}>0.
\end{eqnarray*}%
Therefore the minimal value will occur when $a=c^2/b$. Hence,
\begin{equation}\label{deg1-1.12}%
 -\varphi(c)=\frac{c\sqrt{1+2b^2}}{\sqrt{b^2+2c^4}}. 
\end{equation}%
(\ref{deg1-1.12}) gives the same value as (\ref{deg1-1.10}). Thus this case gives a non-trivial solution (\ref{deg1-1.11a}). Combining our findings, we conclude that a non-trivial solution exists if and only if 
\begin{equation}\label{deg1-1.13}%
\frac{1}{\sqrt{2}}<c\leq 1.
\end{equation}%
Furthermore, this non-trivial solution is the minimum of solutions given by (\ref{deg1-1.9a}) and (\ref{deg1-1.11a}). Since
\begin{eqnarray*}%
 \frac{2\sqrt{2}c}{1+2c^2}-\frac{1}{\sqrt{2}c}=\frac{2c^2-1}{\sqrt{2}c(1+2c^2)}>0\quad {\rm for}\quad \frac{1}{\sqrt{2}}<c\leq 1,
\end{eqnarray*}%
it follows that the minimal solution is 
\begin{equation}\label{deg1-1.14}%
  F_1(c)= \left\{
  \begin{array}{l l}
    1 & \quad {\rm if}\quad 0\leq c\leq \frac{1}{\sqrt{2}} \\
    [ 2 mm]
    \frac{1}{\sqrt{2}c} & \quad {\rm if}\quad \frac{1}{\sqrt{2}} < c \leq 1. \\
  \end{array} \right.
\end{equation}%
see its graph in Figure \ref{fig:5}. To find an extremal pair of polynomials, we take the limit in (\ref{deg1-1.8}) as $b\to 0^+$ and $a=c^2/b \to +\infty$, then we find the limiting function:
\begin{eqnarray*}%
\varphi_{\infty}(z)=-\frac{1}{\sqrt{2}z}.
\end{eqnarray*}%
Taking $p(z)=1$ and $q(z)=\sqrt{2}z$, we obtain an admissible pair $(p,q)=(1,\sqrt{2}z)$, which satisfies (\ref{deg1-1.14}) and therefore it is an extremal pair and it is unique up to a factor $e^{it}$, for some $t\in\mathbb{R}$.
\par
It is interesting to note that the extremal pair does not depend on $c$. We are wondering if the same phenomenon occurs for all degrees $n$ and for Korenblum's problem for analytic functions.

\begin{figure}%
\centering %
\hspace{-4.3cm} %
\begin{minipage}{0.23\linewidth}%
\begin{tikzpicture}
    [inner sep=1mm,
     place/.style={circle,draw=blue!50,
     fill=blue!20,thick},
     scale=4.0]%

\draw [->,>=stealth,black] (0,-0.5) to (0,1.5);%
 \draw [->,>=stealth, black] (-0.5,0)to (1.5,0);%
	\node at (1.5,0) [right] {$c$}; %
	
	 \draw[black, very thick] (0,1)--(0.707,1);%
    \draw [red, very thick] plot[domain=0.707:1] ({\x},{1/(sqrt(2)*\x)});%
		\draw [blue, very thick] plot[domain=0:1,samples=2000] ({\x},{sqrt(1-\x^2)});%
		
     \draw [black, dashed] plot[domain=0.5:0.707] ({\x},{1/(sqrt(2)*\x)});%
     \draw [black, dashed] (0.707,1)--(1,1);%
     \draw [black, dashed] (1,1)--(1,0);%
     \draw [black, dashed] (0.707,1)--(0.707,0);%
		 \draw [black, dashed] (0,0.707)--(1,0.707);%
		
     \node at (0,0) [] {\small{$\bullet$}}; %
		 \node at (0,-0.06) [left] {$0$}; %
     \node at (0.707,0) [] {\small{$\bullet$}}; %
		\node at (0.707,-0.04) [below] {$\frac{1}{\sqrt{2}}$}; %
		\node at (0.33,0) [] {\small{$\bullet$}}; %
		\node at (0.33,-0.04) [below] {\textcolor{magenta}{$\kappa$}}; %
     \node at (1,0) [] {\small{$\bullet$}}; %
		\node at (1,-0.04) [below] {$1$}; %
     \node at (0,1) [] {\small{$\bullet$}}; %
		\node at (0,1) [left] {$1$}; %
     \node at (1,0.707) [] {\small{$\bullet$}}; %
		\node at (0,0.707) [] {\small{$\bullet$}}; %
		\node at (0,0.707) [left] {$\frac{1}{\sqrt{2}}$}; %
		\node at (0.707,1) [] {\small{$\bullet$}}; %
		\node at (0.707, 1.2) [right] {\small{$F_1(c)=\frac{1}{\sqrt{2}c}$}};%
		\node at (0.62,0.62) [left] {\small{$\sqrt{1-c^2}$}};%
		\node at (0.707,0.707) [] {\small{$\bullet$}}; %
 \end{tikzpicture}%
\end{minipage}%
\caption{Graph of $F_1(c)$ and lower bound of $F(c)$}%
\label{fig:5}%
\end{figure}
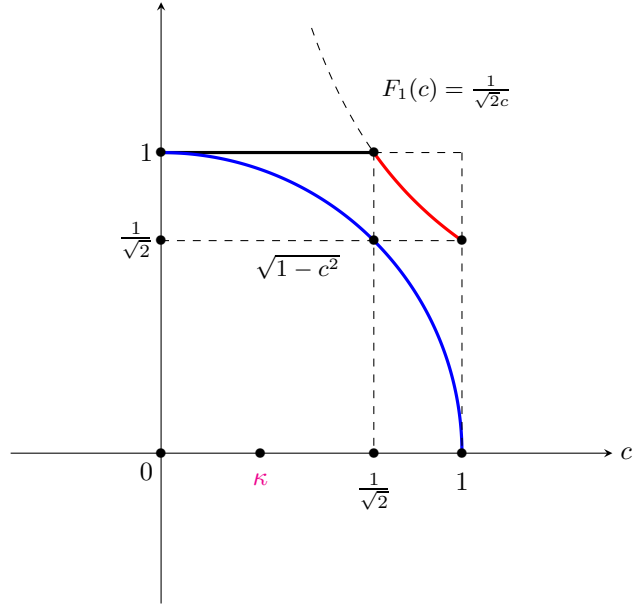%
 

\subsection{Bounded Functions and Blaschke Products}%
\setcounter{equation}{0}  %

We define the following problem for bounded functions:
 \begin{problem}%
 Given $c$, $0<c<1$, find
\begin{equation}\label{eq 3.5a}%
F^b(c)= \inf \limits_{f,g\in S^1 \cap \mathcal{H}^{\infty}} \sup \limits_{c\leq |z|<1} \left |\frac{f(z)}{g(z)} \right |
\end{equation}%
 \end{problem}%
 Note that, $0<F^b(c)\leq 1$. Further, $\kappa^b$ is Korenblum's constant for bounded functions such that $F^b(c)=1$ for $0<c\leq \kappa^b$ and $ F_n(c)<1$ for $\kappa^b<c<1$. Since  $\mathcal{H}^{\infty}\subseteq \mathcal{A}^2(\mathbb{D})$, then $\kappa\leq \kappa^b$. It is interesting to see that the extremal pair of functions are bounded.

\begin{theorem}%
 $\kappa=\kappa^b$.
 \end{theorem}%
 
 \begin{proof}%
 Clearly, $\kappa \leq \kappa^b$. Suppose that $c_n\rightarrow\kappa$ and $c_n<c_{n+1}$. Let $(f_n,g_n)$ be an extremal pair for $F(c_n)$. Then $\|f_n\|_2=\|g_n\|_2=1$ and
 \begin{eqnarray*}%
F(c_n)=\sup \limits_{c_n\leq |z|<1} \left| \frac{f_n(z)}{g_n(z)}\right|\leq 1.
\end{eqnarray*}%
 Therefore,
 \begin{eqnarray*}%
 |f_n(z)|\leq |g_n(z)|\ {\rm for\ all}\ z,\ c_n\leq|z|<1.
 \end{eqnarray*}%
 Note that, for every $n$, there is $\rho_n$, $0<\rho_n<1$, sufficiently close to $1$ such that $\rho_n \rightarrow 1$ as $n\rightarrow\infty$. Define $f_{\rho_n}(z)=f_n(\rho_n z)$, $g_{\rho_n}(z)=g_n(\rho_n z)$. Then
 \begin{eqnarray*}%
 |f_{\rho_n}(z)|\leq |g_{\rho_n}(z)|\ {\rm for\ all}\ z,\ \frac{c_n}{\rho_n}\leq |z|<1. 
\end{eqnarray*}%
Following the steps of Lemma \ref{lem-triple} by replacing $f$, $g$ by $f_n$, $g_n$ respectively and $r$ by $\rho_n$, it is easy to observe that $\|\widetilde{f}_{\rho_n}\|_2=\|\widetilde{g}_{\rho_n}\|_2=1$ and $\widetilde{f}_{\rho_n}$, $\widetilde{g}_{\rho_n}$ are bounded. Therefore, $(\widetilde{f}_{\rho_n}, \widetilde{g}_{\rho_n})$ is admissible for $F^b(c)$ . Since $c_n/\rho_n\to \kappa$ as $n\to\infty$, then $\kappa^b=\kappa$.
 \end{proof}%

For $n\geq 1$, consider the class $B_n$ of Blaschke products of order at most $n$. 
\begin{problem}%
Given $c$, $0<c<1$, find
\begin{equation}\label{eq 3.5b}%
F_n^B(c)= \inf \limits_{f,g\in S^1\cap B_n} \sup \limits_{c\leq |z|<1} \left |\frac{f(z)}{g(z)} \right |.
\end{equation}%
\end{problem}%

Note that, $0<F^B(c)\leq 1$. Further, $\kappa_n^B$ is Korenblum's constant for Blaschke products of order at most $n\geq 1$ such that $F_n^B(c)=1$ for $0<c\leq \kappa^B$ and $ F_n^B(c)<1$ for $\kappa^B<c<1$.  Since $B_n \subseteq B_{n+1}$, it follows that
\begin{eqnarray*}%
0<\kappa \leq \kappa_{n+1}^B \leq \kappa_n^B \leq \ldots \leq \kappa_2^B \leq \kappa_1^B \leq 1.
\end{eqnarray*}%
Following the proof of Lemma \ref{lem-limpol}, one can similarly show that $\kappa =\lim \limits_{n\to\infty} \kappa_n^B$.

\section{``Dual'' Problem \ref{prob 1}}
Given $c$, $0<c<1$, consider the following problem of finding
\begin{equation}\label{eq:prob4}%
F_-(c)=\inf \limits_{f,g\in S^1} \sup \limits_{|z|\leq c} \left | \frac{f(z)}{g(z)}\right |.
\end{equation}%
This problem is dual to the Problem \ref{prob 1} in a sense that the range $c\leq|z|<1$ is replaced by $|z|\leq c$. However, unlike Problem \ref{prob 1}, the solution to problem (\ref{eq:prob4}) is trivial. Precisely, we claim that $F_-(c)\equiv 0$ for $0<c<1$. Fix $0<c<1$ and $\varepsilon>0$, then there is $c<r<1$ and $n\geq 1$ such that $|f_n(z)|^2\leq \varepsilon$ for all $|z|\leq r$. Consider  
\begin{align*}%
 \frac{z^n-r^n}{1-r^n z^n} &= (z^n-r^n)\left [ 1-(rz)^n \right]^{-1}\\
 &= (z^n-r^n)(1+r^n z^n+r^{2n}z^{2n}+r^{3n} z^{3n}+\ldots)\\
 &=-r^n+(1-r^{2n})z^n+r^n(1-r^{2n})z^{2n}+r^{2n}(1-r^{2n})z^{3n}+\ldots\\
 &=-r^n+(1-r^{2n})(z^n+r^n z^{2n}+r^{2n}z^{3n}+\ldots)\\
 \end{align*}%
 Therefore, 
\begin{eqnarray*}%
\left \| \frac{z^n-r^n}{1-r^n z^n} \right \|_2^2=r^{2n}+(1-r^{2n})^2 \sum \limits_{k=1}^\infty \frac{1}{kn+1} r^{2n(k-1)}.
\end{eqnarray*}%
 Define 
\begin{eqnarray*}%
f_n(z) :=\frac{1}{\sqrt{r^{2n}+\frac{(1-r^{2n})^2}{n+1}+(1-r^{2n})^2 \sum \limits_{k=2}^\infty \frac{r^{2n(k-1)}}{kn+1} }}\, \frac{z^n-r^n}{1-r^n z^n},\quad g_n(z)=1.
\end{eqnarray*}%
Define, $\psi(z)=  \frac{z^n-r^n}{1-r^n z^n}$. Then $\psi(z)$ is analytic in the disk $|z|<r$. Therefore, Maximum Modulus Theorem implies that the maximum of $\psi(z)$ occurs on the circle $|z|=r$. Consider 
\begin{eqnarray*}%
|\psi(z)|^2=\frac{2 r^{2n} (1-\cos n\theta)}{1+r^{4n}-2 r^{2n}\cos n\theta}.
\end{eqnarray*}%
Then we want to find $\max \limits_{|z|=r} |\psi(z)|^2$. Note that, 
\begin{eqnarray*}%
\frac{\partial}{\partial \theta}(|\psi(re^{i\theta})|^2)=\frac{2nr^{2n}(1-r^{2n})^2\sin n\theta}{(1+r^{4n}-2 r^{2n}\cos n\theta)^2}.
\end{eqnarray*}%
Therefore, the critical points are located at $\theta=\frac{k\pi}{n}$. It is easy to see that maximum occurs at these points when $k$ is odd. Therefore, 
\begin{eqnarray*}%
\max \limits_{|z|=r} |\psi(z)|^2= \frac{4 r^{2n}}{(1+r^{2n})^2} \Rightarrow \left | \psi(z)\right|^2\leq \frac{4r^{2n}}{(1+r^{2n})^2}< 4r^{2n}.
\end{eqnarray*}%
Thus,
\begin{eqnarray*}%
|f_n(z)|^2 < \frac{4 r^{2n}}{r^{2n}+\frac{(1-r^{2n})^2}{n+1}+(1-r^{2n})^2 \sum \limits_{k=2}^\infty \frac{r^{2n(k-1)}}{kn+1}}= \frac{1}{1+\frac{(1-r^{2n})^2}{r^{2n}(n+1)}+\frac{(1-r^{2n})^2}{r^{2n}} \sum \limits_{k=2}^\infty \frac{r^{2n(k-1)}}{kn+1}}.
\end{eqnarray*}%
Here, $\sum \limits_{k=2}^{\infty} \frac{r^{2n(k-1)}}{kn+1}$ converges for $r<1$. Substituting, $r^2=x$ and $m=nk+1$, we obtain
\begin{eqnarray*}%
\sum \limits_{k=2}^{\infty} \frac{r^{2n(k-1)}}{kn+1}=x^{-(n+1)} \sum \limits_{k=2}^{\infty} \frac{x^m}{m}= r^{-2(n+1)} \left[ -\log(1-r^2)-\sum \limits_{k=1}^{2n} \frac{r^{2k}}{k} \right].
\end{eqnarray*}%
As $n\to\infty$, 
\begin{eqnarray*}%
\frac{(1-r^{2n})^2}{r^{2n}} \sum \limits_{k=2}^{\infty} \frac{r^{2n(k-1)}}{kn+1}= -\frac{(1-r^{2n})^2}{r^{4n+2}} \left[ \log(1-r^2)+\sum \limits_{k=1}^{2n} \frac{r^{2k}}{k} \right] \to +\infty,
\end{eqnarray*}%
and 
\begin{eqnarray*}%
\frac{(1-r^{2n})^2}{r^{2n}(n+1)}\to \frac{(1-0)^2}{0}=+\infty.
\end{eqnarray*}%
Therefore, $|f_n(z)|^2 \to 0$ as $n\to\infty$. 


\section{Further discussion and open questions}%
\setcounter{equation}{0}  %

Though the existence of an extremal pair for both the problems is shown by standard normal family arguments in Section $3$, the question of the uniqueness of the extremal function is still open. It would be interesting to see if an extremal pair $(f,g)$ is  unique for every $c$ in the non-trivial range $(\kappa,1)$ up to some rotation of the functions. Furthermore, a simpler version of the same question would be if $(f_1,g_1)$ and $(f_1,g_2)$ are two extremal pairs for $F(c)$, where $c\in (\kappa,1)$, then can we say $g_2(z)=e^{i\alpha}g_1(z)$ for some real $\alpha$? We have discussed a few of the analytic properties of the function $F(c)$ in Section $4$, we wonder if the function $F(c)$ is uniformly continuous, Lipschitz continuous and differentiable in $(\kappa,1)$.
\par
We obtained $\kappa_1$ for the class of linear polynomials. We can modify the extremal pair for linear polynomials to see if we can guess the extremal pair and in turn improve Korenblum's constant for higher degree polynomials. On that note, let us consider the following two examples based on Problem A.
\par
Consider $f(z)=1$, $g(z)=\sqrt{n}z^{n-1}$. Then $\|f\|_2=\|g\|_2=1$. Further, $\left | f(z)/g(z) \right|=\left| 1/(\sqrt{n}z^{n-1}) \right|<1$ if $|z|^{n-1}>1/\sqrt{n}$ in $c<|z|<1$. This implies that $c> \left(1/n\right)^{1/(2(n-1))}$. Note that the quantity on the right hand side of the last inequality is increasing with respect to $n$ for $n\geq 2$. Therefore, the assumed pair of polynomials does not provide better estimates for Korenblum's constant for polynomials of degree $n\geq 2$ and hence are not extremal. 
\par
Consider $f(z)=\sqrt{n-1}z^{n-2}$, $g(z)=\sqrt{n}z^{n-1}$. Then $\|f\|_2=\|g\|_2=1$. Further, $\left| f(z)/g(z) \right|=\left|\sqrt{(n-1)}/(\sqrt{n}z)\right|<1$ if $|z|>\sqrt{(n-1)/n}$ in $c<|z|<1$. This implies that $c>\sqrt{(n-1)/n}$. Again note that the quantity on the right hand side of the last inequality is increasing with respect to $n$ for $n\geq 2$. Therefore, the assumed pair of polynomials does not provide better estimates for Korenblum's constant for polynomials of degree $n\geq 2$ and hence are not extremal.
\par
Suppose $(p,q)$ is an extremal pair of functions for $F(c)$ in $\mathcal{A}^2(\mathbb{D})$ and $p(z)$, $q(z)$ have a common zero in $A(c,1)$. We conjecture that $p\equiv q$.
\par
One can define Problems A, B and C for $\mathcal{A}^p(\mathbb{D})$ for $p\geq 1$. Then we may ask the following question, for $0<p_1<p_2<\infty$, if $\kappa_{p_1}<\kappa_{p_2}$ or in particular the sequence $\{\kappa_p\}$ is a monotonic sequence for $p\geq 1$. The only fact known so far due to Hinkkanen \cite{H2} is that $\kappa_p\to 1$ as $p\to\infty$. We also wonder in the same direction  that for $0<p_1<p_2<\infty$ and $c \in (\max\{\kappa_{p_1},\kappa_{p_2}\},1)$ if the solution function $F_{p_1}(c)<F_{p_2}(c)$. For $0<p<\infty$ and $c\in(\max\{\kappa,\kappa_p\},1)$, the pair $\{ f_0,g_0\}$ is extremal for the problem
\begin{eqnarray*}
F(c):= \sup_{f,g\in FG(c)} \left( \|f\|_2^2-\|g\|_2^2 \right)
\end{eqnarray*}
Does this imply the pair $\{f_0^{2/p},g_0^{2/p}\}$ is extremal for the following problem?
\begin{eqnarray*}
F_p(c):= \sup_{f,g \in FG_p(c)} \left( \|f\|_p^p-\|g\|_p^p \right)
\end{eqnarray*}


%

\end{document}